\documentclass{amsart}
\usepackage{amssymb, latexsym,framed}
\usepackage{csquotes}
\usepackage{amsthm}
\usepackage{amsmath,amscd}
\usepackage{amsfonts}
\usepackage[dvips]{graphicx}
\usepackage{amsthm}
\usepackage{amsmath,amscd}
\usepackage{amsfonts}
\usepackage[dvips]{graphicx}
\usepackage{xcolor}
\numberwithin{equation}{section} 
\newcommand{\bean}{\begin{eqnarray*}}
\newcommand{\eean}{\end{eqnarray*}}
\newcommand{\be}{\begin{equation}}

\newcommand{\ee}{\end{equation}}
\newcommand{\bd}{\begin{displaymath}}
\newcommand{\ed}{\end{displaymath}}

\newcommand{\beq}{\begin{equation}}
\newcommand{\eeq}{\end{equation}}
\newcommand{\bea}{\begin{eqnarray}}
\newcommand{\eea}{\end{eqnarray}}

\newtheorem{lemma}{Lemma}[section]
\newtheorem{theorem}[lemma]{Theorem}
\newtheorem{definition}[lemma]{Definition}

\newtheorem{remark}[lemma]{Remark}

\newcommand{\norm}[1]{\left\Vert#1\right\Vert}

\newtheorem{proposition}[lemma]{Proposition}

\DeclareMathOperator{\tr}{tr}

\theoremstyle{definition}

\def\({\left(}
\def\){\right)}
\def\1{\mathds{1}}

\def\l|{\left|}

\def\r|{\right|}

\def\XXint#1#2#3{{\setbox0=\hbox{$#1{#2#3}{\int}$ }
\vcenter{\hbox{$#2#3$ }}\kern-.6\wd0}}

\begin{document}

\begin{abstract}
In image and audio signal classification, a major problem is to build stable representations that are invariant under rigid motions and, more generally, to small diffeomorphisms. Translation invariant representations of signals in $\mathbb{C}^n$ are of particular importance. The existence of such representations is intimately related to classical invariant theory, inverse problems in compressed sensing and deep learning. Despite an impressive body of litereature on the subject, most representations available are either: i) not stable due to the presence of high frequencies; ii) non discriminative; iii) non invariant when projected to finite dimensional subspaces. In the present paper, we construct low dimensional representations of signals in $\mathbb{C}^n$ that are invariant under finite unitary group actions, as a special case we establish the existence of low-dimensional and complete $\mathbb{Z}_m$-invariant representations for any $m\in\mathbb{N}$. Our construction yields a stable, discriminative transform with semi-explicit Lipschitz bounds on the dimension; this is particularly relevant for applications. Using some tools from Algebraic Geometry, we define a high dimensional homogeneous function that is injective. We then exploit the projective character of this embedding and see that the target space can be reduced significantly by using a generic linear transformation. Finally, we introduce the notion of {\it non-parallel} map, which is enjoyed by our function and employ this to construct a Lipschitz modification of it.


\end{abstract}

\title{Complete set of translation invariant measurements with Lipschitz bounds}
\date{}
\author{Jameson Cahill \and Andres Contreras \and Andres Contreras Hip}

\address[J.Cahill]{Department of Mathematical Sciences, New Mexico State University, Las Cruces,
New Mexico, USA}
\email{jamesonc@nmsu.edu}

\address[A.Contreras]{Department of Mathematical Sciences, New Mexico State University, Las Cruces,
New Mexico, USA}
\email{acontre@nmsu.edu}

\address[A.Contreras Hip]{Department of Mathematical Sciences, New Mexico State University, Las Cruces,
New Mexico, USA}
\email{albertch@nmsu.edu}

\maketitle

\section{Introduction}

One of the most important problems in machine learning and signal processing is the classification of visual and audio signals, i.e., given an equivalence relation $\sim$ on $\mathbb{C}^n$ one would like to find a map $\Phi:\mathbb{C}^n\rightarrow\mathbb{C}^N$ with the property that $\Phi(x)=\Phi(y)$ if (and ideally only if) $x\sim y$. One would also like $N$ to be as small as possible and for $\Phi$ to be relatively easy to compute. It is a nontrivial task and the study of this an other related problems has sparked a wealth of exciting developments in mathematics in recent years \cite{mallat, afonso, bruna2013invariant, bronstein2017geometric, gu2018} see also \cite{W96} and references therein for more on the history of these types of problems. The sophisticated tools created and used for this purpose, bridge several, seemingly unrelated areas of mathematics.

One type of equivalence relation that comes up in many instances is when the equivalence classes are given by the orbits of some group action, that is, a group $G$ acting on $\mathbb{C}^n$ and $x\sim y$ whenever there is a $g\in G$ for which $x=gy$. In all the applications considered in our study, we will assume that our action comes from a unitary representation of $G$, i.e., we have a group homomorphism 

\begin{equation}\label{assumption unitary}\sigma:G\rightarrow U(n) \mbox{ and for } g\in G \mbox{ and }x\in\mathbb{C}^n, gx=\sigma(g)x .\end{equation}
 In this case we write $\mathbb{C}^n/G=\mathbb{C}^n/\sim$ for the space of orbits (this is a slight abuse of notation as this space depends on the specific representation of $G$) and the quotient metric on $\mathbb{C}^n/G$ is given by (see Lemma 3.3.6 in \cite{metric})
\begin{equation}\label{qmetric}
d_G([x],[y])=\inf_{g\in G}\|x-gy\|
\end{equation}
where $[x]$ denotes the orbit of $x$ under the action of $G$.

By some elementary considerations, using small separating sets, one can show the existence of an injective map \[ f :\mathbb{C}^n/ G \to \mathbb{C}^{2n + 1} \] that distinguishes orbits. This is essentially a restatement of Proposition 5.1.1 in \cite{emilie}, see Section \ref{alg}.
While this looks very good in principle because it solves the classification problem with a small number of measurements, it is not good enough for applications. In general one would like to obtain maps that are not only injective but also stable in the sense that the distortion of the representation is globally controlled. In addition, one would like to construct measurements that are robust enough to preserve injectivity under small perturbations, due to noise for example. To be more precise, one would want a bi-Lipschitz map $f$, that is a map for which there are constants $0<c\leq C<\infty$ so that
\begin{equation}\label{lulips}
cd_G([x],[y])\leq \|f([x])-f([y])\|\leq Cd_G([x],[y])
\end{equation}
for every $x,y\in\mathbb{C}^n .$

 In this work we study the classification problem under finite unitary actions of signals in $\mathbb{C}^n.$ We obtain a representation in low dimensions, the first of its kind in a nontrivial setting. Its simplicity makes it a good option for applications. We first obtain some general results that apply to arbitrary unitary actions of finite groups before specializing to the important case of discrete translations (cyclic groups) where our results can be made much more explicit. 

One of the motivations of our study is the problem of finding a representation of functions in $L^2(\mathbb{R}^n)$ that is invariant under translations. One such representation is given by the modulus of the Fourier transform; however, it is not injective and furthermore it is known to yield instabilities in the presence of high frequencies and it is for this reason not suitable for implementation. The numerous works that have tackled different aspects of this problem in recent years rely mostly on statistical learning techniques and as such they inherit some of their problems and limitiations.To develop a more satisfactory approach to classification of translation invariant signals in $L^2(\mathbb{R}^n)$, Mallat introduced \cite{mallat} a \textit{scattering transform} that is invariant under translations but more importantly, comes with a global upper-Lipschitz bound (the scattering transform is non-expansive). To deal with the problem of instabilities created by the presence of high frequencies, Mallat tames the contributions of fine scale oscillations by using a wavelet-based convolutional network; a limiting procedure gives the desired object at infinite depth. Although the scattering transform provides a non-expansive map, invariant under general small diffeomorphisms, it is not discriminative and is not actually invariant in finite dimensions (invariance is only achieved in the limit, while for implementation one needs to cut off the process after a few iterations). In applications one always works in large, but finite dimensions and thus an alternative to the scattering transform is desirable, one that can be readily applied in this setting with few measurements and with a guarantee of no misclassification. The goal of our work is to provide a different approach to the group invariant classification problem in the relevant finite dimensional setting.

Of particular interest is the case where $G=\mathbb{Z}_m$ is a cyclic group. The representation of $G$ in this case will simply be the powers of some matrix $T$ that satisfies $T^m=I$. Since the minimal polynomial of $T$ divides $x^m-1$ it follows that $T$ is diagonalizable and all of its eigenvalues are $m$th roots of unity. To better understand complete translation invariant measurements in this setting, we further specialize to $G=\mathbb{Z}_n$ where the action is given by

\be \label{Znaction}T^kx(j)=x(j-k \text{ mod }n)\mbox{ for }k=1,...,n\ee 

It is well known $T$ is diagonalized by the discrete Fourier transform $\mathcal{F}x=\hat{x}$, and so in the Fourier domain we have
\[
\widehat{Tx}(j)=e^{2\pi i j/n} \hat{x}(j)
\]
Theorem \ref{main} guarantees the existence of a Lipschitz map that distinguishes orbits under this finite translation action with a number of measurements that is linear on the dimension.

Another important application that is closely related to the problem of phase retrieval is the case where $T=\omega I$ where $\omega$ is an $m$th root of unity, see Section \ref{hg}. 

A consequence of our study gives a stable and discriminative nonlinear transform that is $\mathbb{Z}_m$-invariant. More precisely, we prove the following:

\begin{theorem}\label{main} There is a $\mathbb{Z}_m$-invariant map $\Phi:\mathbb{C}^n \mapsto \mathbb{C}^{2n+1}$ that induces an injective map $\tilde{\Phi} : \mathbb{C}^n /\mathbb{Z}_m \mapsto \mathbb{C}^{2n + 1} ,$and a constant $C>0$ depending only on $m$ such that

\begin{equation} \| \Phi (x)- \Phi (y) \|\leq C d_{\mathbb{Z}_m}([x],[y]) ,
\end{equation}
 for every $ x, y \in \mathbb{C}^n .$
\end{theorem}

We provide more analytic and geometric information about the map $\Phi$ in Section \ref{proofofmain}, in particular we give explicit bounds on the optimal Lipschitz constant $C.$

\vskip.2in
\noindent{\bf Constructing complete invariant representations with few measurements}
\vskip.1in

As mentioned, in this paper we follow a different approach to build translation invariant representations. Our perspective hinges upon the use of polynomial measurements (as oposed to an infinite chain of convolution-taking modulus operations). The alternative we propose here has the following advantages

\begin{itemize}
\item it lends itself well for applications; this is because the measurement and coding of signals is usually treated in very high but finite dimensional spaces. The calculations involved in the transform are very easy to handle and do not require complex mathematical or computational operations.
\item the number of measurements needed to differentiate signals in different orbits is linear in the number of dimensions.
\item the map we construct is actually $G$-invariant, as opposed to approximately invariant. As such, it is an accurate representation at the chosen level of precision(scale).
\item the representation really solves the classification problem, it is truly injective.
\item not only is the map produced stable, but it comes with almost explicit bounds. 
\end{itemize}
Our representation is not entirely constructive since the argument relies on choosing a linear map in the dimension reduction part and this is done using an abstract result. This in itself does not limit the range of applicability of the transform which works for a generic choice. 

The construction of an initial large set of polynomial measurements that allow us to separate orbits follows from simple and intuitive observations. However, this is just the starting point as a map consisting of such monomials does not satisfy any of the other properties we ask of our representations.

To better illustrate the ideas and challenges behind the proof of Theorem \ref{main}, let us restrict ourselves to the case where $G$ is $\mathbb{Z}_p,$ where $p$ is a prime number and the action is given by translation as in \eqref{Znaction}. We let $M$ be the modulation operator which is the diagonalization of the translation operator, that is, $M = \text{diag } {\left( e^{2 \pi i j / p} \right)}_{j = 0, \hdots , p-1}$. In the Fourier domain we can define 
\[
 F(\widehat{x}) := (\widehat{x}_0, {\widehat{x}_1}^p, {\widehat{x}_2}^p,\hdots, {\widehat{x}_{p-1}}^p, {\widehat{x}_1}^{p-2} \widehat{x}_2, \hdots, {\widehat{x}_1}^{p-k} \widehat{x}_k, \hdots, \widehat{x}_1 \widehat{x}_{p-1}).
\]
One can see that $F(\widehat{x})=F(\widehat{y})$ if and only if $\widehat{x}=M^k  \widehat{y},$ for some $k\in \mathbb{N}.$ Indeed, suppose first that $F(\widehat{x}) = F(\widehat{y}),$ where $\widehat{x} = (\widehat{x}_0, \hdots, \widehat{x}_{p-1})$ and $\widehat{y} = (\widehat{y}_0, \hdots, \widehat{y}_{p-1}).$ Then $\widehat{x}_0 = \widehat{y}_0$ and there are $n_k, k = 2, \hdots, p,$ such that $\widehat{x}_k = {\zeta}^{n_k} \widehat{y}_k \mbox{ for } k = 2, \hdots, p,$ where $\zeta = e^{2 \pi i /p}.$ On the other hand, we have 
\[
{(\zeta ^{n_1})}^{p-k} \zeta ^{n_k} = 1, \mbox{ and so } {\zeta}^{n_k} = ({\zeta}^{n_1})^{k}.
\]
But then, $\widehat{x} = M^{n_1} \widehat{y} .$  

However, one can readily see that this map does not separate orbits, because if $x_1 = 0,$ then $x_2, \hdots, x_{p-1}$ are completely free. Also, since the measurements are polynomials, we cannot expect to have global Lipschitz bounds. To solve the separating problem, we can add a monomial for each pair of variables, but then we have $\mathcal{O} (p^2)$ measurements. We will see in Section \ref{abstractresults} that by taking generic linear combinations of these measurements we can reduce the dimension dramatically. Then the problem becomes how to turn the resulting map into a Lipschitz one without losing the discriminative property and without adding more measurements. It is at this point that a geometric property we introduce (see definition \ref{non-parallel prop.}) becomes crucial for building a stable transform. We combine this property satisfied by our polynomials together with the fact that our actions are unitary to reduce all our measurements to the unit sphere.

Although the maps we construct are injective, we show in Section \ref{absenceofllb} that their inverses are almost never Lipschitz. This of course does not rule out the possibility that other representations may come with a lower Lipschitz bound. However we believe such representations must be non-algebraic and therefore essentially different from ours.

The paper is organized as follows. In the second section we introduce some necessary algebraic background and we also discuss the connection with the well-known problem of phase retrieval. In Section \ref{abstractresults} we prove some general results concerning the construction of effective low dimensional discriminative measurements under natural assumptions. In Section \ref{proofofmain} we verify that the examples of interest for us satisfy the conditions of the theorems in Section \ref{abstractresults} and we obtain explicit Lipschitz bounds for the resulting transforms. We finish the paper with a discussion on open questions and a few remarks in Section \ref{remarks}.

\section{Background}
\subsection{Algebraic invariants}\label{alg}

Given an action of a group $G$ on $\mathbb{C}^n$ we denote by $\mathbb{C}[x]^G$ the ring of polynomials in $\mathbb{C}[x]=\mathbb{C}[x_1,...,x_n]$ that are invariant under this action. It is well known that this ring is finitely generated as a $\mathbb{C}$-algebra, however the generating set could be arbitrarily large. Nonetheless one might hope that given a generating set $\{p_i\}_{i=1}^N$ then the map $P:\mathbb{C}^n\rightarrow\mathbb{C}^N$ given by
\begin{equation}\label{algmap}
P(x)=(p_i(x))_{i=1}^N
\end{equation}
would induce a map $\tilde{P}$ that is injective on $\mathbb{C}^n/G$. Unfortunately this turns out not to be true in general. In fact, an example for this is $\mathbb{C}^* ,$ the multiplicative group of $\mathbb{C}$ acting on $\mathbb{C}^2$ via scalar multiplication \cite{DerksenKemper}, example 2.3.1.

As in \cite{finitegroups}, we define
\[
\nu_{sep} = \{ (x, y) | \mbox{ for all } P \in {\mathbb{C} [x]}^G, \mbox{ we have } P(x) = P(y) \}.
\]
Then a set $S \subset {\mathbb{C} [x]}^G$ is said to be \textit{separating} if whenever $P(x) = P(y)$ for all $P \in S,$ we have that $(x, y) \in \nu_{sep}$. It is known that separating sets can exist which are much smaller than the number of generators of $\mathbb{C}[x]^G$, in fact in \cite{emilie} the following is proved in Theorem 5.1.1:

\begin{theorem}
If $G$ acts on $\mathbb{C}^n$ then a separating set of size at most $2n+1$ exists.
\end{theorem} 

Furthermore, in \cite{DerksenKemper}, section 2.3, it is noted that if $G$ is a finite group action, then 
\[
\nu_{sep} = \{ (x, g x) | x \in \mathbb{C}^n, g \in G\}.
 \]

We now specialize to the case where $G=\mathbb{Z}_m .$ Without loss of generality we can assume that the action of $G$ is given by the powers of a diagonal matrix $T=\text{diag}(t_1,...,t_n)$ where $t_i$ is an $m$th root of unity for every $i$. It is proved in Theorem 5.2.1 in \cite{emilie} that the following set of $n(n+1)/2$ monomials is a separating set:
\begin{equation}\label{monomials}
\mathcal{P}_T=\{x_i^{m_i},x_j^{a_{jk}}x_k^{b_{jk}}:1\leq i\leq n,1\leq j<k\leq n\}
\end{equation}
where $m_i|m$ and $t_i$ is a primitive $m_i$th root of unity, and $a_{jk}$ is minimal such that there exists a $b_{jk}<m_k$ with $x_j^{a_{jk}}x_k^{b_{jk}}$ invariant. From the set of measurements \eqref{monomials}, an invariant map $F_G : \mathbb{C}^n \to \mathbb{C}^{n(n + 1)/2}$ defined by
\begin{equation} \label{F_G}
F_T (x) = ({x_1}^{m_1}, \hdots , {x_n}^{m_n}, {\{x_j^{a_{jk}}x_k^{b_{jk}}\}}_{1 \leq j, k \leq n}),
\end{equation}
induces an explicit injective map $\tilde{F}_T:\mathbb{C}^n/\mathbb{Z}_m\rightarrow\mathbb{C}^{n(n+1)/2}$.  We will reduce the dimension of the target space by showing that a suitably generic linear map $\ell:\mathbb{C}^{n(n+1)/2}\to \mathbb{C}^{2n+1}$ is injective when restricted to the image of $F_T.$ In general, this should be the optimal dimension, but we know already that for specific examples this can be reduced even further to $2n-1$ as can be seen from the special cases covered by \cite{emilie}, Proposition 5.2.2 (see also the discussion on subsection \ref{PR} about real phase retrieval being a particular case of this problem).

\subsection{Phase retrieval}\label{PR}

Phase retrieval is the problem of recovering a signal $x$ in $\mathbb{C}^n$ (or $\mathbb{R}^n$) up to a global phase factor from a collection of intensity measurements $(|\langle x,\varphi_i\rangle|^2)_{i=1}^N$. This type of problem comes up in many applications and has a rich history, but has seen considerable interest in the last decade or so since the publication of \cite{BCE}. To state the problem in the setting of this paper let $\mathbb{T}=\{\lambda\in\mathbb{C}:|\lambda|=1\}$ denote the one dimensional torus and let $\mathbb{T}$ act on $\mathbb{C}^n$ by scalar multiplication. Then given any collection of vectors $\{\varphi_i\}_{i=1}^N$ the mapping $\Phi:\mathbb{C}^n\rightarrow\mathbb{R}^N$ given by
$$
\Phi(x)=(|\langle x,\varphi_i\rangle|^2)_{i=1}^N
$$
is invariant under the action of $\mathbb{T}$ so we can consider the induced map $\tilde{\Phi}$ whose domain is $\mathbb{C}^n/\mathbb{T}$. The first problem in phase retrieval is to understand when this map is injective.

Let $\mathbb{H}_n$ denote the space of $n\times n$ Hermitian matrices and note that $\mathbb{H}_n$ is a vector space over the real numbers (not the complex numbers) of dimension $n^2$. The Hilbert-Schmidt inner product on $\mathbb{H}_n$ is given by
$$
\langle S,T\rangle_{HS}=\tr(ST).
$$
Now consider the map from $\mathbb{C}^n$ to $\mathbb{H}_n$ given by $x\mapsto xx^*$. First note that $xx^*=yy^*$ if and only if $x=\lambda y$ for some $\lambda\in\mathbb{T}$, so this map is injective on $\mathbb{C}^n/\mathbb{T}$, and the image of this map is the set $\mathcal{S}$ of positive rank one matrices in $\mathbb{H}_n$ (which looks like $\mathbb{P}^{n-1}\times\mathbb{R}_+$). Next observe that for $x,y\in\mathbb{C}^n$
\begin{eqnarray*}
\langle xx^*,yy*\rangle_{HS}&=&\tr(xx^*yy^*) \\
&=&\tr(y^*xx^*y) \\
&=&|\langle x,y\rangle|^2.
\end{eqnarray*}
Given a collection of vectors $\{\varphi_i\}_{i=1}^N\subseteq\mathbb{C}^n$ define the linear map $\tilde{\Phi}:\mathbb{H}_n\rightarrow\mathbb{R}^N$ given by
$$
\tilde{\Phi}(S)=(\tr(S\varphi_i\varphi_i^*))_{i=1}^N
$$
and note that $\Phi(x)=\tilde{\Phi}(xx^*)$, so $\Phi$ is injective on $\mathbb{C}^n/\mathbb{T}$ if and only if $\tilde{\Phi}$ is injective when restricted to $\mathcal{S}$. If $\tilde{\Phi}$ is not injective on $\mathcal{S}$ then there are $xx^*\neq yy^*$ so that $\tilde{\Phi}(xx^*)=\tilde{\Phi}(yy^*)$ and so $xx^*-yy^*\in\ker(\tilde{\Phi})$. From this observation it is straightforward to prove the following (see Lemma 9 in \cite{BCMN14}):

\begin{lemma}\label{prinj}
$\tilde{\Phi}$ is injective on $\mathcal{S}$ if and only if every nonzero matrix in $\ker(\tilde{\Phi})$ has rank at least 3.
\end{lemma}

Since the dimension of the set of rank at most 2 $n\times n$ Hermitian matrices is $4n-4$ this led the authors of \cite{BCMN14} to conjecture that $N\geq 4n-4$ was necessary for $\Phi$ to be injective. If this were taking place over the complex numbers this would follow directly from the Projective Dimension Theorem, in \cite{CEHV15} the conjecture was proven for infinitely many values of $n$, however in \cite{V15} a counterexample is constructed with $n=4$ and $N=11$.

In most applications the measurements $\Phi(x)$ will never be exact and will be corrupted by noise of some form. Therefore we we would like the map to be not just injective but bi-Lipschitz as defined in \eqref{lulips} (here we use the quotient metric on $\mathbb{C}^n/\mathbb{T}$), however, it is shown in \cite{BCMN14} that this $\Phi$ can never be bi-Lipschitz in this sense. There are (at least) two ways of dealing with this situation. The first, as done in \cite{BCMN14}, is to modify the map to get a new map that is bi-Lipschitz. Another alternative which is explored in \cite{BD16} is to replace the quotient metric with a different metric on $\mathbb{C}^n/\mathbb{T}$ with respect to which $\Phi$ is bi-Lipschitz. In this paper we will encounter a similar situation where we will have an initial $G$-invariant map which is injective but not Lipschitz. Our approach will be to exploit a geometric property of this map to produce a new map which is still injective but also Lipschitz.

We can also study phase retrieval over the real numbers where we replace $\mathbb{C}^n$ with $\mathbb{R}^n$ and $\mathbb{T}$ with $\{1,-1\}$, which corresponds to the $\mathbb{Z}_2$ action on $\mathbb{R}^n$ given by $-I$. In this case the analysis above is still valid, but for $x\in\mathbb{R}^n$ we have that $xx^*=xx^T$ is real and symmetric, and the entries of $xx^T$ are precisely the monomials in $\mathcal{P}_{-I}$ (see \eqref{monomials}). Therefore, from this perspective real phase retrieval can be thought of as a very special case of the type of cyclic group actions that we consider in this paper.

\section{Main results}\label{abstractresults}
In this section we present a series of abstract results that yield a discriminative and stable representation $\Phi .$ It can be seen in the next section that the hypotheses needed for the existence of such $\Phi$ are satisfied by our objects of interest; we actually believe that the abstract framework here provided can find further applications in signal processing.   The construction of the map $\Phi$ rests on discriminative polynomial measurements in some possibly high dimensional space. These are later mapped to an $\mathcal{O}(n)$ dimensional space via a generic linear map. Still, this map is not necessarily Lipschitz so we appeal to its geometric properties to modify it so that it becomes Lipschitz while preserving injectivity. In the rest of this section we will present the steps described.
\subsection{Dimension reduction}
We already have that the polynomial map $F_T$ defined in \eqref{F_T} is separating. We look at the problem of reducing the dimension, and getting a Lipschitz bound. This has to be done carefully because we need to reduce the dimension while still preserving injectivity and finally controlling the distortion. The next theorem reduces the dimension.

\begin{theorem}\label{dimreduction} Let $G$ act on $\mathbb{C}^n$ and suppose $P:\mathbb{C}^n\rightarrow\mathbb{C}^N$ is a polynomial $G$-invariant map such that the induced map $\tilde{P}:\mathbb{C}^n/G\rightarrow\mathbb{C}^N$ is injective. Then for $k\geq 2n+1,$ $\ell\circ\tilde{P}$ is injective for a generic linear map $\ell:\mathbb{C}^N\rightarrow\mathbb{C}^k$.
\end{theorem}
\begin{proof}First, since the components of $P$ are polynomials, we can write $P = (P_1, P_2, \hdots, P_N),$ and 
\[
P_i = \sum_{j = 1}^{c_i} p_{i,j}, 
\]
where $p_{i,j}$ is a monomial of degree $d_{i,j}.$ Then one of these monomials achieves maximum degree $d.$ Now we produce a new map $F(x, y, t)=(F_1,...,F_N)$ as follows:
\begin{eqnarray*}
f_{i,j}(x,y,t)&=&t^{d-d_{i,j}}(p_{i,j}(x)-p_{i,j}(y)) \\
F_i &=& \sum_{j = 1}^{c_i} f_{i,j}.
\end{eqnarray*} 

Note that $F$ is homogeneous and regular (it is a polynomial map), and so we know that $\mathrm{Im}(F)\subseteq\mathbb{C}^N$ is a projective variety. Therefore 
\[
\dim(\mathrm{Im}(F))\leq 2n + 1.
\]
For a linear map $\ell:\mathbb{C}^N\rightarrow\mathbb{C}^k$ if $\ell\circ\tilde{P}$ is not injective then there are $x,y\in\mathbb{C}^n$ so that $F(x,y,1)\neq 0$ but $F(x,y,1)\in\ker(\ell)$. We now claim that if $k\geq 2n+1$ then for a generic $\ell$ we have $\ker(\ell)\cap\mathrm{Im}(F)=\{0\}$ which would prove the theorem. To this end let
$$
S = \{ [\ell], [w] :0\neq w \in \mathrm{Im}(F), \ell w = 0, [\ell] \in \mathbb{P}(\mathbb{C}^{k\times N}), [w] \in \mathbb{P} (\mathbb{C}^N) \}
$$ 
(here $[\cdot]$ is a class in projective space ). It is easy to see that $S$ is projective. Let us note that we can assume without loss of generality that $S$ is irreducible(if not, reason like below on each irreducible component). Also let $\pi_1, \pi_2$ be the projections of $S$, that is: 
\[
\pi_1 ([\ell], [w]) = [\ell],  \pi_2 ([\ell], [w]) = [w].
\]
Then $\pi_2 (S) = [\mathrm{im}(F)],$ therefore, $\dim (\pi_2 (S)) = \dim (\mathrm{im}(F)) - 1\leq 2n.$ By \cite{Harris} corollary 11.13 we know that if we take $[w_0] \in \mathbb{P}^N ,$ then 
\[
\dim (S) = \dim (\pi_2^{-1} [w_0]) + \dim (\pi_2 (S)).
\]
 On the other hand, we know
\[
\dim (\pi_2^{-1} ([w_0])) = \dim (\{\ell \in \mathbb{C}^{k\times N} | \ell w_0 = 0\}) - 1 = k(N-1)-1 = kN - k -1
\]
This implies that
$$
\dim (S)\leq kN-k-1+2n.\\
$$
Next observe that
\begin{eqnarray*}
\pi_1(S)&=&\{[\ell]:\ker(\ell)\cap\mathrm{Im}(F)\neq\{0\}\} \\
&=&\{[\ell]:\ell\circ\tilde{P}\text{ is not injective}\}.
\end{eqnarray*}
Therefore if $\dim(\pi_1(S))<\dim(\mathbb{P}(\mathbb{C}^{k\times N}))=kN-1$ then $\ell\circ\tilde{P}$ is injective for a generic $\ell$. Finally,
$$
\dim(\pi_1(S))\leq\dim(S)=kN-k-1+2n,
$$
so we require
$$
kN-k-1+2n<kN-1
$$
which means $k>2n.$
\end{proof}

The main idea in the proof of Theorem \ref{dimreduction} is very similar to that of Lemma \ref{prinj} in that we want to show that a generic linear map of appropriate rank is injective when restricted to a particular algebraic variety which means that the kernel of the linear map needs to avoid differences of pairs of vectors that are on the variety. The main difference between these two results is that the varieties under consideration in the phase retrieval case they have a lot of structure, in particular they are projective, whereas ours are not. On the other hand, phase retrieval takes place over the real numbers (even in the complex case) which complicates the use of certain tools from algebraic geometry. We note that variants of these types of arguments have been used in other recovery problems \cite{WX17,RWX17,CEHV15}.

\subsection{Non-parallel maps induce Lipschitz invariant representations}\label{Lipschitz map}
As we anticipated, we will discuss other general theorems that let us make modifications to turn a map into a Lipschitz one, provided the map we started with satisfies a geometric condition. To that end, we introduce the following concept:
\begin{definition}\label{non-parallel prop.} Suppose $G$ acts on $\mathbb{C}^n$ and $F:\mathbb{C}^n\rightarrow\mathbb{C}^N$ is $G$-invariant. We say $F$ has the non-parallel property if the following holds: If $\|x\|=\|y\|=1$ and $F(x) = \lambda F(y)$ for some $\lambda > 0,$ then $x=gy$ for some $g\in G$.
\begin{remark} \label{NPPpreserved} Note that if $F:\mathbb{C}^n \to \mathbb{C}^N$ satisfies the non-parallel property, and $\ell : \mathbb{C}^N \to \mathbb{C}^m$ is a linear map such that
\[
\ker(\ell) \cap (\mathrm{Im} F - \mathrm{Im} F) = \{0\},
\]
then $\ell \circ F$ also satisfies the non-parallel property.
\end{remark}
\end{definition}
The following will be used to get a candidate for a Lipschitz map.
\begin{definition} For any map $F:\mathbb{C}^n \to \mathbb{C}^N,$ we define $\Phi_F$ by 
\begin{equation}\label{finalmap}
\Phi_F(x):= 
\left\{
\begin{matrix}
\norm{x} F\left( \frac{x}{\norm{x}}\right) & \mbox{, if } x \neq 0 \\
0 & \mbox{, if } x = 0.
\end{matrix}
\right.
\end{equation}
\end{definition}
\begin{proposition} \label{finalmapinjectivity}
Suppose $G$ acts on $\mathbb{C}^n$ according to \eqref{assumption unitary} and $F:\mathbb{C}^n\rightarrow\mathbb{C}^N$ is a $G$-invariant map. Then $\Phi_F$ as defined in \eqref{finalmap} satisfies \\
(a) $\Phi_F$ is also $G$-invariant, \\
(b) If $F$ has the non-parallel property and if the induced map $\tilde{F}:\mathbb{C}^n/G\rightarrow\mathbb{C}^N$ is injective then the corresponding induced map $\tilde{\Phi}_F$ is also injective.
\end{proposition}
\begin{proof} Throughout the proof, we write $\Phi$ instead $\Phi_F$.

\noindent (a) Note that if $x \in \mathbb{C}^n$ and $g \in G,$ then
\[
\Phi (gx) = \|gx\| F \left( \frac{gx}{\|gx\|} \right) = \|x\| F \left( \frac{1}{\|x\|} gx \right) = \|x\| F\left(g \left( \frac{x}{\|x\|}\right) \right)
\] 
Where we have used our general assumption \eqref{assumption unitary}. Since $F$ is invariant, we have that
\[
\Phi (gx) = \|x\| F \left( \frac{x}{\|x\|}\right) = \Phi (x).
\]
So $\Phi$ is invariant.\\
(b) Suppose $x, y \in \mathbb{C}^n $ are such that
\[
\Phi (x) = \Phi (y).
\]
If $x = 0,$ then 
\[
0 = \Phi (x) = \Phi (y),
\]
which implies that $y = 0 = x.$ Now if $x \neq 0,$ we have that 
\[
\|x\| F\left(\frac{x}{\|x\|}\right) = \|y\| F\left(\frac{y}{\|y\|}\right),
\]
So in particular $F\left(\frac{x}{\|x\|}\right)$ and $F\left(\frac{y}{\|y\|}\right)$ are parallel. Since $F$ satisfies the non-parallel property, we have that $\frac{x}{\|x\|} =g\left( \frac{y}{\|y\|}\right)$ for some $g\in G$. This implies that
\[
F\left(\frac{x}{\|x\|}\right) = F\left(\frac{y}{\|y\|}\right),
 \]
and since $\tilde{F}$ is injective we conclude $x=gy$ which means $\tilde{\Phi}$ is also injective.
\end{proof}

\begin{remark}
Note that in the proof of (b) we really show that when $F$ has the non-parallel property then $\tilde{\Phi}_F(x)=\tilde{\Phi}_F(y)$ if and only if $\tilde{F}(x)=\tilde{F}(y)$ regardless of whether or not $\tilde{F}$ is injective.
\end{remark}

In what follows, $\mathbb{S}$ will denote the unit sphere in $\mathbb{C}^n.$

\begin{theorem} \label{generalthm} Let $G$ be a unitary group acting on $\mathbb{C}^n$ according to \eqref{assumption unitary}. Let $H:\mathbb{C}^n\to \mathbb{C}^N$ be a $G$-invariant map satisfying the non-parallel property. Assume $H$ is $C^1 .$ Then the map $\Phi_H : \mathbb{C}^n \to \mathbb{C}^{N}$ defined in \eqref{finalmap} is Lipschitz. Furthermore, the optimal Lipschitz constant $\alpha$ such that $\|\tilde{\Phi}_H ([x]) - \tilde{\Phi}_H ([y])\| \leq \alpha d_G([x],[y])$ satisfies $\alpha \leq 3C,$ where
\[
C = \max \left\{ \norm{\nabla_{\mathbb{S}} H}_\infty, \max_{z \in \mathbb{S}} \|H(z)\|\right\}.
\]
\end{theorem}
\begin{proof}
We will drop the subscript $H$ in order to ease the notation. Let $x, y \in \mathbb{C}^n.$ If $x = 0$ and $y \neq 0,$ then
\[
\norm{\Phi (0) - \Phi (y)} = \|y\| \norm{H\left( \frac{y}{\|y\|}\right)} \leq \max_{ \|x\| = 1} \|H(x)\| \|y\| = \max_{ \|x\| = 1} \|H(x)\| d_{G} ([y], [0]).
\]
Now if $x, y \neq 0,$ then 
\begin{eqnarray*}
\norm{\Phi (x) - \Phi (y)} &=& \norm{\|x\| H\left(\frac{x}{\|x\|}\right) - \|y\| H\left(\frac{y}{\|y\|}\right)} \\
  & \leq & \|x\| \norm{H\left(\frac{x}{\|x\|}\right) - H\left(\frac{y}{\|y\|}\right)} + \vert \|x\| - \|y\| \vert \norm{H\left(\frac{y}{\|y\|}\right)}. 
\end{eqnarray*}
Since $H$ is bounded on the sphere, we have that
\[
\vert \|x\| - \|y\| \vert \norm{H\left(\frac{y}{\|y\|}\right)} \leq \|x - y\| \max_{z \in \mathbb{S}} \|H(z)\|.
\]
This implies that
\[
\norm{\Phi (x) - \Phi (y)} \leq \|x\| \norm{H\left(\frac{x}{\|x\|}\right) - H\left(\frac{y}{\|y\|}\right)} + \|x - y\| \max_{z \in \mathbb{S}} \|H(z)\|.
\]
By symmetry on $x$ and $y,$ we have that
\[
\norm{\Phi (x) - \Phi (y)} \leq \min\left\{ \|x\|, \|y\| \right\} \norm{H\left(\frac{x}{\|x\|}\right) - H\left(\frac{y}{\|y\|}\right)} + \|x - y\| \max_{z \in \mathbb{S}} \|H(z)\|.
\]
We know that $H$ is differentiable when restricted to the unit sphere $\mathbb{S}$. Suppose $x, y \in \mathbb{C}^n.$ Therefore we can use the mean value theorem to find 
\[
\norm{H\left(\frac{x}{\|x\|}\right) - H\left(\frac{y}{\|y\|}\right)} \leq  \norm{\nabla_{\mathbb{S}} H}_\infty \norm{\frac{x}{\|x\|} - \frac{y}{\|y\|}}.
\]
Hence
\[
\min\left\{ \|x\|, \|y\| \right\} \norm{H\left(\frac{x}{\|x\|}\right) - H\left(\frac{y}{\|y\|}\right)} \leq \min\left\{ \|x\|, \|y\| \right\} \norm{\nabla_{\mathbb{S}} H}_\infty \norm{\frac{x}{\|x\|} - \frac{y}{\|y\|}}.
\]
Then if we let 
\[
C = \max \left\{ \norm{\nabla_{\mathbb{S}} H}_\infty, \max_{z \in \mathbb{S}} \|H(z)\| \right\},
\]
we will have that
\[
\|\Phi (x) - \Phi (y)\| \leq C \left( \min\left\{ \|x\|, \|y\| \right\} \norm{\frac{x}{\|x\|} - \frac{y}{\|y\|}} + \|x - y\| \right).
\]
If $\|y\| \leq \|x\|,$ then
\begin{eqnarray}
\min \{ \|x\|, \|y\|\} \norm{\frac{x}{\|x\|} - \frac{y}{\|y\|}}& = & \|y\| \norm{\frac{x}{\|x\|} - \frac{y}{\|y\|}}\nonumber\\
                                                                                         & = & \norm{\frac{\|y\| x}{\|x\|} - y}\nonumber\\
                                                                                         & \leq & \|x - y\| + \norm{\frac{\|y\| x}{\|x\|} - x}\nonumber\\
                                                                                         & = & \|x - y\| + \left| \frac{\|y\|}{\|x\|} - 1 \right| \|x\|,\nonumber
\end{eqnarray}
from where we obtain
\[
\min\left\{ \|x\|, \|y\| \right\} \norm{\frac{x}{\|x\|} - \frac{y}{\|y\|}} \leq 2\|x - y\|.
\]
This implies that 
\[
\|\Phi (x) - \Phi (y)\| \leq 3C \|x - y\|.
\]
Since $\Phi$ is invariant, we know that for all $g \in G,$ we have that
\[
\| \Phi (x) -\Phi (y) \| = \| \Phi (x) - \Phi (gy) \| \leq 3C \|x - gy\|,
\]
therefore
\[
\| \Phi (x) - \Phi (y) \| \leq 3C \min_{g \in G} \|x - gy\| = 3C d_G([x],[y]).
\]
This gives the Lipschitz bound.
\end{proof}

 As mentioned in the introduction, an ideal representation should not only be injective but also control the distance between classes in terms of a lower Lipschitz bound. Because the representation we construct is made of linear combination of polynomials with no particular structure (as opposed to phase retrieval where the measurements take the particular form $|\langle x, \phi_k\rangle |^2_{k=1, \ldots , M}$), we cannot expect to have a lower Lipschitz bound at our level of generality. Moreover, we will see in the next section that when $G=\mathbb{Z}_m , m\geq 3.$ the map $\Phi$ is not bi-Lipschitz.

In phase retrieval the map $\ell$ sends the measurements to $(\mathbb{R}_+)^N$ for some $N,$ and since the map is homogeneous of degree 2, one can take the square root of each component without destroying separation while making the map bi-Lipschitz. Our separating measurements are essentially complex valued-the phases contain essential information one cannot ignore, and so we cannot appeal to a similar idea.

\section{Applications to cyclic groups}\label{proofofmain} 
In this section, we discuss a few applications of the theorems in the previous section. Recall that in the case of $G=\mathbb{Z}_m$ we can assume the action is given by the powers of a diagonal matrix $T=\mathrm{diag}(t_1,...,t_m)$ where each $t_i$ is an $m$th root of unity.
\begin{theorem} \label{cyclic case}
Suppose we have the group $\mathbb{Z}_m$ acting on $\mathbb{C}^n,$  then we have an injective map 
\[
\tilde{\Phi}:\mathbb{C}^n / \mathbb{Z}_m \to \mathbb{C}^{2n+1}.
\]
Moreover, $\tilde{\Phi}$ has a Lipschitz constant $3 m \|\ell \|.$
\end{theorem}
This theorem will be a consequence of the following lemmas:
\begin{lemma} \label{F_G NPP}
$F_T$ defined in \eqref{F_G} satisfies the non-parallel property.
\end{lemma}
\begin{proof}
Suppose we have $x, y \in \mathbb{C}^n$ and $\lambda > 0$ such that
\[
F_T (x) = \lambda F_T (y).
\]
Let $\omega$ be the first $m$th root of unity and
\[
\tilde y = (y_1 {\lambda}^{\frac{1}{m_1}}, y_2 {\lambda}^{\frac{1}{m_2}}, \hdots, y_k {\lambda}^{\frac{1}{m_k}}, \hdots, y_n {\lambda}^{\frac{1}{m_n}}).
\]
We will prove that $\lambda F_T (y) = F_T (\tilde{y}).$ Let $S = \{ i \,:\, y_i \neq 0\}.$ Note that 
\[
\lambda {y_k}^{m_k} = {({\lambda}^{\frac{1}{m_k}} y_k)}^{m_k} = {{\tilde y}_k}^{m_k}.
\]
To prove that $\lambda y_j^{a_{jk}}y_k^{b_{jk}} = \tilde y_j^{a_{jk}}\tilde y_k^{b_{jk}},$ there are two cases:\\
\underline{Case 1:} One of $j, k$ is not in $S.$\\
In this case, one of $y_j, y_k$ is $0$ (by definition of $S$) and by the definition of $\tilde{y}_j,$ one of $\tilde{y}_j, \tilde{y}_k$ is $0.$ Therefore, 
\[
\lambda y_j^{a_{jk}}y_k^{b_{jk}} = 0 = \tilde y_j^{a_{jk}}\tilde y_k^{b_{jk}}.
\] 
\underline{Case 2:} Both $j, k \in S.$\\
In this case, we know $y_j, y_k \neq 0.$ We have $\lambda {y_j}^{m_j} = {x_j}^{m_j},$ and so
\[
{\left(\frac{x_j}{y_j}\right)}^{m_j} = \lambda.
\]
Then ${\left(\frac{x_j}{{\lambda}^{\frac{1}{m_j}}y_j}\right)}^{m_j} = 1.$ Since $m_j | m,$ there is a $p_j$ such that $\frac{x_j}{{\lambda}^{\frac{1}{m_j}}y_j} = {\omega}^{p_j},$ which implies that 
\[
\frac{x_j}{y_j} = {\omega}^{p_j} {\lambda}^{\frac{1}{m_j}}.
\]
We know that $\lambda y_j^{a_{jk}}y_k^{b_{jk}} = x_j^{a_{jk}}x_k^{b_{jk}},$ which implies that
\[
\lambda =  {\left(\frac{x_j}{y_j}\right)}^{a_{jk}}{\left(\frac{x_k}{y_k}\right)}^{b_{jk}}, 
\]
and since $\frac{x_j}{y_j} = {\omega}^{p_j} {\lambda}^{\frac{1}{m_j}}, \frac{x_k}{y_k} = {\omega}^{p_k} {\lambda}^{\frac{1}{m_k}},$ we have
\[
\lambda = {\left({\omega}^{p_j} {\lambda}^{\frac{1}{m_j}}\right)}^{a_{jk}} {\left({\omega}^{p_k} {\lambda}^{\frac{1}{m_k}}\right)}^{b_{jk}} = {\lambda}^{\frac{a_{jk}}{m_i} + \frac{b_{jk}}{m_k}} {\omega}^{p_j a_{jk} + p_k b_{jk}}.
\]
Taking the modulus on both sides yields $\lambda = {\lambda}^{\frac{a_{jk}}{m_j} + \frac{b_{jk}}{m_k}},$ hence 
\[
\lambda y_j^{a_{jk}}y_k^{b_{jk}} = {\lambda}^{\frac{a_{jk}}{m_j} + \frac{b_{jk}}{m_k}} y_j^{a_{jk}}y_k^{b_{jk}} = {y_j {\lambda}^{\frac{1}{m_j}}}^{a_{jk}} {y_k {\lambda}^{\frac{1}{m_k}}}^{b_{jk}} = \tilde y_j^{a_{jk}}\tilde y_k^{b_{jk}}.
\]
Therefore, 
\[
\lambda F_T (y) = F_T (\tilde y),
\]
which implies that $x \sim \tilde y.$ Thus for some $1 \leq k \leq n$ we have that 
\[
\tilde y = T^k x.
\] 
Since T is unitary, we have that
\[
\sum_{j = 1}^n {\vert y_j\vert}^2 = {\|x\|}^2 = {\|T^k x\|}^2 = {\|\tilde y\|}^2 = \sum_{j = 1}^n {\lambda}^{\frac{2}{m_j}} {\vert y_j\vert}^2.
\]
Since $y \neq 0$ (in fact $\|y\| = 1$), we know that $\sum_{j = 1}^n {\lambda}^{\frac{2}{m_j}} {\vert y_j\vert}^2$ is increasing in $\lambda.$ Hence $\lambda = 1.$ This implies that $\Phi (x) = \Phi (\tilde y) = \Phi (y).$ Since $\tilde{\Phi}$ is injective, $x \sim y.$ This proves the non-parallel property.\\
\end{proof}
\begin{lemma} The map $H = \ell \circ F_T ,$ where $F_T:\mathbb{C}^n \to \mathbb{C}^{n(n + 1)/2}$ is the map defined in \eqref{F_G}  and $\ell$ is any map satisfying the conclusion of Theorem \ref{dimreduction}, satisfies the non-parallel property. In particular, the map $\Phi$ defined in \eqref{finalmap} for this choice of $H,$ induces an injective map $\tilde{\Phi} .$ Furthermore, $\tilde{\Phi}$ satisfies the Lipschitz bound
\[
\|\tilde{\Phi} ([x]) - \tilde{\Phi} ([y])\| \leq 3C d_G([x],[y]), \mbox{ with }\\
C = \|\ell\| \max \left\{ \norm{\nabla_{\mathbb{S}} (F_T)}_\infty, \max_{z \in \mathbb{S}} \|F_T(z)\|\right\},
\]
where $\|\ell\|$ is the operator norm of $\ell.$
\end{lemma}
\begin{proof}
Combining remark \ref{NPPpreserved}, theorem \ref{dimreduction}, and proposition \ref{F_G NPP}, we obtain that $\ell \circ F_T$ satisfies the non-parallel property. Now, using proposition \ref{finalmapinjectivity}, we get
\begin{eqnarray*}
\|\Phi (x) - \Phi (y)\| \leq 3C d_G([x],[y]), \mbox{ with }
C &=& \max \left\{ \norm{\nabla_{\mathbb{S}} (\ell \circ F_T)}_\infty, \max_{z \in \mathbb{S}} \|\ell \circ F_T(z)\| \right\}\nonumber\\ &\leq &\|\ell\| \max \left\{ \norm{\nabla_{\mathbb{S}} (F_T)}_\infty, \max_{z \in \mathbb{S}} \|F_T(z)\| \right\}
\end{eqnarray*}
\end{proof}
\begin{lemma}\label{constants estimate} If $F_T$ is defined as in \eqref{finalmap}, then 
\begin{equation}
\max \left\{ \|\nabla F_T\|_\infty, \|F_T\|_\infty\right\} \leq m.
\end{equation}
\end{lemma}
\begin{proof}
Note that 
\begin{eqnarray}
\partial_{x_k} F_T & = & \partial_{x_k} (\{{x_i}^{m_i}\}_{1 \leq i \leq m}, \{{x_j}^{a_{jl}} {x_l}^{b_{jl}}\}_{1 \leq j < k \leq m}) \nonumber \\
 & = & (\{m_i {x_i}^{m_i - 1} \delta_{i,k}\}_{1 \leq i \leq m}, \{a_{jl} {x_j}^{a_{jl} -1} {x_l}^{b_{jl}} \delta_{j,k} + b_{jl} {x_j}^{a_{jl}} {x_l}^{b_{jl} - 1} \delta_{l,k}\}_{1 \leq j < l \leq m})
\end{eqnarray}
But if $x \in \mathbb{S},$ then
\begin{eqnarray}
\sup_{1 \leq k \leq m} \|\partial_{x_k} F_T\|_{\infty} & \leq & \max \left\{ \max_{1 \leq i \leq m} \left\{ m_i {x_i}^{m_i - 1} \right\}, \max_{1 \leq j < l \leq m} \left\{a_{jl} {x_j}^{a_{jl} -1} {x_l}^{b_{jl}}\right\}, \max \left\{ b_{jl} {x_j}^{a_{jl}} {x_l}^{b_{jl} - 1}\right\} \right\} \nonumber\\
 & \leq & \max_{1 \leq i, j, l \leq m} \left\{ m_i, a_{jl}, b_{jl}\right\}.
\end{eqnarray}
Since 
\[
\|\nabla_{\mathbb{S}} F_T \| \leq \|\nabla F_T\|,
\]
we have that
\[
\|\nabla_{\mathbb{S}} F_T\|_\infty \leq \max_{1 \leq i, j, l \leq m} \left\{ m_i, a_{jl}, b_{jl} \right\} \leq m.
\]
\end{proof}
\textit{Proof of Theorem \ref{cyclic case}.} 

We start with the map $F_T : \mathbb{C}^n \to \mathbb{C}^{n (n + 1)/2}$ given by \eqref{F_G}. It is proven that this map is separating in \cite{emilie}, chapter 5. By Lemma \ref{F_G NPP}, it satisfies the non-parallel property. We can apply Theorem \ref{dimreduction} to reduce the dimension of the target space to $2n + 1.$ For a generic $\ell : \mathbb{C}^{n (n + 1)/2} \to \mathbb{C}^{2n + 1},$ we have that
\[
\ell \circ F_T : \mathbb{C}^n \to \mathbb{C}^{2n + 1}
\]
is a separating map. Now, by Remark \ref{NPPpreserved}, we have that $\ell \circ F_T$ also satisfies the non-parallel property. Hence we can use Theorem \ref{generalthm} for $\ell \circ F_T,$ and we obtain an injective map $\tilde{\Phi}:\mathbb{C}^n/\mathbb{Z}_m \to \mathbb{C}^{2n + 1}$ with a Lipschitz bound:
\[
\|\tilde{\Phi} (x) - \tilde{\Phi} (y)\| \leq 3C d_{\mathbb{Z}_m} (x,y)
\]
where 
\[
C = \|\ell\| \max \left\{ \norm{\nabla_{\mathbb{S}} (F_T)}_\infty, \max_{z \in \mathbb{S}} \|F_T (z)\| \right\}.
\]
Using Lemma \ref{constants estimate}, we directly deduce that
\[
\|\tilde{\Phi}(x) - \tilde{\Phi}(y)\| \leq 3 \|\ell\| m d_{\mathbb{Z}_m} (x,y).
\]
\qed \\
Before proceeding we make a few remarks.

Since the set of invariants we use in this case are monomials of the form $x_i^ax_j^b$ then it could make sense to arrange these in an $n\times n$ matrix. In fact, since we only have one monomial for each $(i,j)$ we can choose to make this matrix triangular, symmetric, or Hermitian if we so desire. As such this problem fits naturally into the broader context of matrix recovery problems. In this paper we do not need to exploit this point of view as our proofs do not benefit from it.

We now return to the original motivation of this work, which was to study translation invariant measurements on $\mathbb{C}^n$. In this case $\mathbb{Z}_n$ acts on $\mathbb{C}^n$ via translation as in \eqref{Znaction}. In this case the matrix $T$ is not diagonal, but it is diagonalizable by the Fourier transform. In the Fourier domain our action is given by powers of the modulation matrix
$$
M=\mathrm{diag}(e^{2\pi i j/n})_{j=0}^{n-1}.
$$
One of the main motivations that Mallat cites in \cite{mallat} for the development of the scattering transform is that although the modulus of the Fourier transform is translation invariant, it is not stable. In our context by taking the modulus of the first $n$ entries of $F_T(x)$, i.e., the measurements of the form $x_i^{m_i}$, we can recover the modulus of the Fourier transform, but even these are not enough to get injectivity which is why we need the rest of the measurements.

\subsection{The homogeneous case}\label{hg}
In this subsection we will make a few remarks about the case when  $F_T$ is homogeneous. We begin by showing that this only happens when $T=\omega I$ for some root of unity $\omega$. Note that the case $ \omega = -1$ which corresponds to real phase retrieval is a special case of this.

\begin{proposition}
$F_T$ is homogeneous if and only if $T=\omega I$ for some $m$th root of unity $\omega$.
\end{proposition}
\begin{proof}
If $T=\omega I$ then 
$$
\mathcal{P}_T=\{x_i^m,x_ix_j^{m-1}\}
$$
and so $F_T$ is homogeneous of degree $m$.

Conversely, suppose $F_T$ is homogeneous of degree $m$ and $T=\mathrm{diag}(t_1,...,t_n)$. We readily see that each $t_i$ must be a primitive $m$th root of unity, if not we would have at least on pair of monomials of the form $x_i^{m_i}$ and $x_j^{m_j}$ with $m_i\neq m_j$. Now observe that the monomial $x_i^ax_j^b$ must satisfy $a+b=m$. Then, since $t_i$ and $t_j$ are both primitive $m$th roots of unity there is a $k$ so that $t_i=t_j^k$. Now by the division algorithm and the minimality of $a$ we see that $a= m\text{ mod } k$ and $a+kb=m$. This implies $k=1$ and therefore $t_i=t_j=\omega$.
\end{proof}

Note that in the proof of Theorem \ref{dimreduction} we needed to introduce a new variable to homogenize the map $F_T$, however when $T=\omega I$ this is not necessary. This means we can slightly improve the conclusion of Theorem \ref{main} as follows:

\begin{theorem} \label{omega I case}
Let $\omega$ be an $m$-th root of unity and let $\mathbb{Z}_m$ act on $\mathbb{C}^n$ via $T=\omega I$. Then there exists a $\Phi: \mathbb{C}^n \to \mathbb{C}^{2n}$ satisfying the same conclusions as Theorem \ref{cyclic case}.
\end{theorem}

As we noted above we can arrange our monomials into an $n\times n$ matrix. In this case one such matrix would be
$$
x^{m-1}x^T
$$
where $x^{m-1}$ denotes the vector whose components are the components of $x$ raised to the power $m-1$. Note that this is a rank one matrix. This means that just as in the case of phase retrieval, any linear map $\ell$ that has the property that every nonzero matrix in $\ker(\ell)$ has rank at  least 3 will satisfy the conclusion of Theorem \ref{omega I case}. Therefore we can use such an $\ell$ to obtain the $\tilde{\Phi}$ in theorem \ref{cyclic case}. However, the set of rank (at most) two $n\times n$ matrices is a determinantal variety and is well known to have dimension $4n-4$, so by the Projective Dimension Theorem any linear map whose kernel avoids this variety must have rank at least $4n-4$ whereas Theorem \ref{dimreduction} says a generic linear map of rank $2n$ will yield the desired conclusion. This is because we only need $\ker(\ell)$ to avoid rank two matrices of the form $x^{m-1}x^T-y^{m-1}y^T$ which is a much smaller subvariety.

\subsection{No lower Lipschitz bounds}\label{absenceofllb} A natural question is whether or not the map $\Phi$ satisfies a lower Lipschitz bound. We show that we can not expect this to be the case in general. The next proposition shows, $\Phi^{-1}$ is never Lipschitz when $G$ is cyclic and $n, |G|\geq 3.$

\begin{proposition} Let $m,n\geq 3.$ Let $T$ be a representation of $\mathbb{Z}_m$ acting on $\mathbb{C}^n.$ Then

\[\inf_{x,y\in \mathbb{C}^n} \frac{\|\Phi_{F_T} (x)-\Phi_{F_T} (y)\|}{d_{\mathbb{Z}_m} ([x],[y])}=0.\]

\end{proposition}

\begin{proof}

Recall the definition of the map $F_T$ \eqref{F_G}. Without loss of generality we may assume that $m_i\geq 2$ for all $i=1,\ldots, n.$ Since $m,n \geq 3,$ we see that there is a pair $(i,j)$ with $\max \{a_{ij}, b_{ij}\}\geq 2.$ Again, without loss of generality, we assume $(1,2)$ is such a pair and that $a_{12}\geq 2.$

Let $x=(0,1,0,0,\ldots, 0)$ and for $\varepsilon>0$ small consider the points $ x_\varepsilon=( \varepsilon, \sqrt{1-\varepsilon^2},\ldots, 0).$ We know that for sufficiently small $\varepsilon,$ $d_{\mathbb{Z}_{m}}([x_\varepsilon],[x])=\|x_\varepsilon-x\|,$ and therefore

\begin{equation}\label{nolower1}
d_{\mathbb{Z}_{m}}([x_\varepsilon],[x])=\varepsilon (1+o(1)).
\end{equation}

On the other hand, note that $\Phi$ is equal to $\ell\circ F_T$ on the unit sphere by definition. Also, any nonzero component of $\ell\circ F_T (x_\varepsilon)$ is a linear combination of $x_1^{m_1}, x_1^{a_{12}}x_2^{b_{12}}, x_2^{m_2}.$ If $k$ is any such component, then

\[\left( \ell\circ F_T (x_\varepsilon)-\ell\circ F_T (x) \right)_k=c_k^1 \varepsilon^{m_1}+c_k^2 \varepsilon^{a_{12}}\left( 1-\varepsilon^2\right)^{b_{12}/2}+c_k^3 \left( 1-\varepsilon^2 \right)^{m_2/2}-c_k^3 ,\]

which is $\mathcal{O}(\varepsilon^2)$ because $m_1, a_{12}\geq 2.$ This implies that

\begin{equation}\label{nolower2}
\|\Phi(x_\varepsilon)-\Phi(x)\|=\mathcal{O}(\varepsilon^2)
\end{equation}

The conclusion follows from \eqref{nolower1} and \eqref{nolower2}.

\end{proof}

To illustrate the previous proposition, consider the example 5.2.1. in \cite{emilie}: there $\mathbb{Z}_{12}$ acts on $\mathbb{C}^5$ and one has the explicit $\ell\circ F_T: \mathbb{C}^5\mapsto \mathbb{C}^8$ given by:

\[(x_1,\ldots, x_5) \mapsto (x_5^6, x_4 x_5^5, x_4^6+x_3 x_5^4, x_3 x_4^4+x_2^2 x_5^3, x_3^3+x_2^2 x_4^3+x_1 x_5^3, x_1 x_4^3, x_1 x_2^2, x_1^2).\]

In this example we can take $x=(0, 0, 0, 1,0)$ and $ x_\varepsilon=(0,0,0,\sqrt{1-\varepsilon^2}, \varepsilon).$

Although $\tilde{\Phi}^{-1}$ is not Lipschitz, it is always continuous.

\begin{proposition} Under the same hypotheses of Proposition \ref{finalmapinjectivity}, the map $\tilde{\Phi}_F^{-1}$ is continuous.

\end{proposition}

\begin{proof} Let $x\in \mathbb{C}^n$ and $(x_k)_{k\in \mathbb{N}}\subseteq \mathbb{C}^n$ be such that $\Phi_F (x_k)\to \Phi_F (x), \mbox{ as }k\to \infty.$ Because $F$ is bounded away from zero on the sphere, we see that $\|x_k\|$ is bounded. Therefore, up to subsequence (not relabelled) $x_k$ converges to some $y\in \mathbb{C}^n .$ By continuity of $\Phi_F ,$ we see that $\Phi_F (y)=\Phi_F (x)$ which then implies $y \sim x .$ We have shown that every subsequence of $([x_k])_{k\in \mathbb{N}}$ possesses a subsequence converging to $[x],$ thus the whole sequence converges to $[x].$

\end{proof}

\section{Conclusions, open questions and final remarks}\label{remarks}

As it can be seen from our analysis, effective $G$-invariant representations can be built in finite dimensions by exploiting underlying algebraic and geometric properties of polynomial invariants. Though our transforms are general enough to cover many problems of interest, there are still some natural open questions regarding the construction of complete measurements in the finite dimensional setting. 

A first question has to do with the use of polynomial invariants. In the end our transform does not preserve any of the features or the the algebraic structure of the maps leading up to it, so it would be very enlightening to explore ways to bypass the use of the invariant polynomials, and to appeal to the geometric non-parallel property to produce the final transform. A more analytical and more flexible approach would probably yield embeddings into even lower dimensions with much more explicit controls.

About the dimension reduction as we perform it here, it is certainly worth studying how to make a constructive choice of a linear map $\ell$ than that provided by Theorem \ref{dimreduction}. Even though our maps do not have a lower Lipschitz bound the choice of $\ell$ should play a significant role in any computation of the inverse map. More specifically, since we know that $\mathrm{ker}(\ell)$ must avoid nonzero vectors that are differences of elements of $\mathrm{Im}(F_T)$ in order for $\tilde{\Phi}$ to be injective, then it seems intuitively clear that having these vectors ``bounded away'' from $\mathrm{ker}(\ell)$ in some sense should provide numerical stability. There are a variety of ways we could define "bounded away" in this context. We expect properties similar to the nullspace property \cite{D06} and the restricted isometry property \cite{CRT06} to be of use here.

Another question that has to do with the algebraic approach employed here is how to treat more general group actions. Our assumption that $G$ is a finite group is essential for the separating polynomials to be actually separating in the sense that they discriminate orbits (see section \ref{alg} for more details).  Also, another assumption that one should try to do without is that the action is unitary. This assumption is certainly convenient due to the handy representation and particular set of separating monomials, but it also plays a role in the rest of the construction and it affects the Lipschitz bound in a significant way because without it $\Phi$ would not be invariant. Thus, it should be clear that a different perspective is needed to treat more general situations.

One salient desired property that is absent in our construction is a quantitative control of the injectivity of $\Phi,$ for example our map does not come with a lower Lipschitz bound. Of course, this could be a matter of our use of polynomial invariants, but it could be that a much more delicate problem is at hand. We believe that even if complete sets of measurements could be constructed that make use of other types of invariants that would have a lower Lipschitz constant, the bounds obtained could be very bad and would probably go to $0$ as $n\to\infty.$ This is known to be the case in phase retrieval \cite{cahill2016phase}.

Finally, one of the motivations behind this work was to provide an alternative to the scattering transform of Mallat \cite{mallat} and the strategy preferred in the works \cite{bruna2013invariant, bronstein2017geometric, rohe2017svf, liu2014deeply, eppenhof2018deformable, cirecsan2010deep} based on neural networks; our main goal was to provide an approach better adapted to deal with finite dimensional problems. In \cite{mallat} the transform obtained is non expansive, that is the Lipschitz constant is equal to $1.$ This is equivalent to, in our setting, having a bound independent of the dimension. We note that in principle, we could obtain non expansive maps simply by normalizing $\Phi$ (which depends on $n$) by the corresponding Lipschitz constant (which also depends on $n$). The challenge becomes then to understand the possible limits of these normalized transforms as $n\to\infty$ and how they relate to the scattering transform. The authors anticipate studying some of these problems in the future.

\bigskip

\noindent \textbf{Acknowledgements.} The work of A. Contreras was partially supported by a grant from the Simons Foundation \# 426318.

\bibliographystyle{plain}
\bibliography{sep}
\end{document}